\documentclass[11pt,reqno,a4paper]{amsart}

\usepackage{etoolbox}
\usepackage[T1]{fontenc}
\usepackage[USenglish]{babel}
\usepackage[utf8]{inputenc}
\usepackage{lmodern}
\usepackage{graphicx}
\usepackage[numbers]{natbib}
\usepackage{csquotes}
\usepackage{amsmath,amssymb,mathtools,amsthm,thmtools}
\usepackage{microtype}
\usepackage{tikz,tikz-cd}
\usepackage{enumitem}
\usepackage[colorlinks=true,allcolors=blue,pdfusetitle=true,hypertexnames=true,pdfpagelabels,plainpages=false,hyperindex]{hyperref}
\usepackage{chngcntr}
\usepackage{natbib}

\usepackage[hyphenbreaks]{breakurl}

\AtBeginDocument{
	\counterwithout{figure}{section}
}
\patchcmd{\subsection}{-.5em}{.5em}{}{}
\patchcmd{\subsubsection}{-.5em}{.5em}{}{}

\tikzset{
	every picture/.style = {line width=0.8pt},
	vertex/.style = {shape = circle, inner sep=1.35pt, fill=black}
}
\newcommand{\hgline}[3][very thick]{
\pgfmathsetmacro{\thetaone}{#2}
\pgfmathsetmacro{\thetatwo}{#3}
\pgfmathsetmacro{\theta}{(\thetaone+\thetatwo)/2}
\pgfmathsetmacro{\phi}{abs(\thetaone-\thetatwo)/2}
\pgfmathsetmacro{\close}{less(abs(\phi-90),0.0001)}
\ifdim \close pt = 1pt
    \draw[thick,#1] (\theta+180:1) -- (\theta:1)
\else\
    \pgfmathsetmacro{\R}{tan(\phi)}
    \pgfmathsetmacro{\distance}{sqrt(1+\R^2)}
    \draw[thick,#1] (\theta:\distance) circle (\R);
	\path (\theta:\distance) -- (\theta:\distance-\R-.08)
\fi
}

\newcommand{\ag}[1]{\mathbf{#1}}
\newcommand{\cat}[1]{\textnormal{#1}}
\newcommand{\term}[1]{\emph{#1}}
\newcommand{\R}{\mathbb{R}}
\newcommand{\C}{\mathbb{C}}
\newcommand{\Q}{\mathbb{Q}}
\newcommand{\Z}{\mathbb{Z}}
\newcommand{\N}{\mathbb{N}}
\newcommand{\congrightarrow}{\xrightarrow{\raisebox{-1pt}[0ex][0ex]{$\scriptscriptstyle\cong$}}}
\renewcommand{\O}{\mathcal{O}}
\renewcommand{\H}{\mathbb{H}}
\newcommand{\Hamilton}{\mathcal{H}}
\DeclareMathOperator{\SL}{SL}
\DeclareMathOperator{\GL}{GL}
\DeclareMathOperator{\SO}{SO}
\DeclareMathOperator{\rank}{rank}
\DeclareMathOperator{\res}{res}
\DeclareMathOperator{\Res}{Res}
\newcommand{\dR}{\operatorname{dR}}
\newcommand{\dRc}{\operatorname{dR},\operatorname{c}}
\newcommand{\product}{\textnormal{prod}}
\newcommand{\emb}{\textnormal{emb}}
\newcommand{\cent}{\textnormal{cent}}
\renewcommand{\a}{\mathfrak{a}}
\renewcommand{\b}{\mathfrak{b}}

\renewcommand{\AA}[1]{\mathbb{A}_{f,#1}}
\newcommand{\OO}[1]{\mathbb{O}_{f,#1}}

\declaretheorem[numberwithin=section]{theorem}
\declaretheorem[numberlike=theorem]{proposition}
\declaretheorem[numberlike=theorem]{corollary}
\declaretheorem[numberlike=theorem]{lemma}
\declaretheorem[numberlike=theorem,style=definition]{definition}
\declaretheorem[numberlike=theorem,style=definition]{remark}
\declaretheorem[numberlike=theorem,style=definition]{example}

\usepackage[capitalize,noabbrev]{cleveref}

\title[Geometric construction of homology classes]{Geometric construction of homology classes in Riemannian manifolds covered by products of hyperbolic planes}
\thanks{The author acknowledges funding by the Deutsche Forschungsgemeinschaft (DFG, German Research Foundation) - 281869850 (RTG 2229).}
\author{Pascal Zschumme}
\address{Karlsruhe Institute of Technology, 76131 Karlsruhe, Germany}
\email{pascal.zschumme@kit.edu}


\begin{document}

\vspace*{5ex} 

\begin{abstract}
We study the homology of Riemannian manifolds of finite volume that are covered by an $r$-fold  product $(\H^2)^r = \H^2 \times \ldots \times \H^2$ of hyperbolic planes.
Using a variation of a method developed by Avramidi and Nguyen-Phan, we show that any such manifold $M$ possesses, up to finite coverings, an arbitrarily large number of compact oriented flat totally geodesic $r$-dimensional submanifolds whose fundamental classes are linearly independent in the homology group $H_r(M;\Z)$.
\end{abstract}

\maketitle

\section{Introduction} \label{sec:introduction}

Let $M$ be a Riemannian manifold of finite volume that is covered by $(\H^2)^r = \H^2 \times \ldots \times \H^2$. If $r=1$, then $M$ is a hyperbolic surface and its homology is well understood. Otherwise, $M$ can be a complicated object. For example, let $d > 0$ be a square-free integer and consider the real quadratic field $F = \Q(\sqrt{d})$ with its two distinct embeddings $\sigma_1,\sigma_2 \colon F \hookrightarrow \R$. Let $\O_F$ be the ring of integers of $F$. Then the group $\SL_2(\O_F)$ acts properly discontinuously on the product $\H^2 \times \H^2$ by
\[
\gamma \cdot (z_1,z_2) := (\sigma_1(\gamma) \cdot z_1, \sigma_2(\gamma) \cdot z_2),
\]
where $\sigma_i(\gamma) \cdot z_i$ is the action of $\SL_2(\R)$ on $\H^2$ by fractional linear transformations.
For any torsion-free subgroup of finite index $\Gamma \subset \SL_2(\O_F)$, the quotient $\Gamma \backslash (\H^2 \times \H^2)$ 
is a Riemannian manifold of finite volume that is covered by $\H^2 \times \H^2$. It is called a \term{Hilbert modular surface} and is an irreducible locally symmetric space of higher rank.

The homology of such a locally symmetric space is in general hard to compute,
and even if one can do so, the geometric meaning of the homology classes is
often lost during the computation. We choose a more geometric approach going
back to Millson \cite{Millson}, in which one studies homology classes that are the fundamental classes of totally geodesic submanifolds.

Promising candidates for such submanifolds are the compact flat totally geodesic submanifolds of dimension equal to the rank of the locally symmetric space. It is known that these submanifolds exist in any nonpositively curved locally symmetric space of finite volume (see \cite[Proposition 5.1]{Pettet}). Furthermore, Pettet and Souto proved in \cite[Theorem 1.2]{Pettet} that they are \term{non-peripheral}, which means they cannot be homotoped outside of every compact subset of the locally symmetric space.
This suggests that these submanifolds might contribute to the homology of the locally symmetric space. Avramidi and Nguyen-Phan \cite{FlatCycles} have investigated this question for the locally symmetric space $M = \SL_n(\Z) \backslash \SL_n(\R) / \SO(n)$ and proved that, in fact, up to finite coverings, the compact oriented flat totally geodesic submanifolds of dimension equal to the rank of $M$ contribute to the free part of the homology group of $M$.

We prove the following theorem, which shows that this is also true for all Riemannian manifolds of finite volume that are covered by products of hyperbolic planes:

\begin{theorem} \label{thm:main}
Let $M$ be a Riemannian manifold of finite volume that is covered by $(\H^2)^r$. Then for any $n \in \N$, there exists a connected finite covering $M' \to M$ and compact oriented flat totally geodesic $r$-dimensional submanifolds $\mathcal{F}_1,\ldots,\mathcal{F}_n \subset M'$ such that the images of the fundamental classes $[\mathcal{F}_1],\ldots,[\mathcal{F}_n]$ in $H_r(M';\Z)$ are linearly independent.
\end{theorem}

In particular, this implies that the $r$th Betti number of $M$ can be made arbitrarily large by going to a finite covering space of $M$. We remark that this fact was already known. It follows from the non-vanishing of the $r$th $L^2$-Betti number of $M$ (see \cite[p.715]{Growth-L2-Invariants}) and L\"uck's approximation theorem. However, our result shows that the corresponding homology classes can be realized geometrically, which, to our knowledge, is a new result.

Our proof proceeds as follows: Using induction on $\dim(M)$, one finds irreducible Riemannian manifolds $M_1,\ldots,M_k$ for some $k \in \N$ and a finite covering $M_1 \times \ldots \times M_k \to M$. An application of the K\"unneth theorem for homology now shows that it suffices to prove \cref{thm:main} for irreducible manifolds.
So assume that $M$ is irreducible.
If $r = 1$, then $M$ is a hyperbolic surface and we can find a finite covering surface $M'$ of $M$ whose genus is at least $n$. The surface $M'$ then has $n$ distinct simple closed geodesics whose homology classes are linearly independent in $H_1(M';\Z)$. This proves \cref{thm:main} for $r=1$.

On the other hand, if $r > 1$, then Margulis' arithmeticity theorem implies that $M$ is \term{arithmetic}, which means that it is finitely covered by a quotient of $(\H^2)^r$ by an arithmetically defined lattice in $\SL_2(\R)^r$.

The goal of this article is to explain our proof of \cref{thm:main} for arithmetic manifolds.
It is structured as follows: In \cref{sec:algebraic-groups}, we fix our notation for algebraic groups, discuss arithmetically defined lattices, and state Margulis' arithmeticity theorem. In \cref{sec:quaternion-algebras}, we describe the arithmetically defined lattices in $\SL_2(\R)^r$ using quaternion algebras.
In \cref{sec:flats}, we study flats in symmetric spaces, and in \cref{sec:geometric-cycles} we discuss geometric cycles.
Finally, in \cref{sec:geometric-construction}, we describe our construction of the covering $M' \to M$ and the submanifolds $\mathcal{F}_1, \ldots, \mathcal{F}_n \subset M'$ for an arithmetic manifold $M$. This construction is based on the techniques developed by Avramidi and Nguyen-Phan in \cite{FlatCycles}.

The material covered in this article evolved from the author's doctoral thesis \cite{ZschummePhD}. There, the interested reader can find more details on our construction and the proof of \cref{thm:main}.

\section{Algebraic groups and Margulis' arithmeticity theorem} \label{sec:algebraic-groups}

We consider algebraic groups as special cases of group schemes and identify them with their functors of points (see \cite{Waterhouse,Milne-Book}). 
By this, we mean the following: Let $R$ be a commutative ring.
A \term{group scheme over $R$} is a functor $\ag{G} \colon \cat{Alg}_R \to \cat{Grp}$
from the category of commutative $R$-algebras to the category of groups whose composition with the forgetful functor $\cat{Grp} \to \cat{Set}$ is representable by a finitely generated $R$-algebra. We denote this algebra by $\O(\ag{G})$.
One can think of $\ag{G}$ as a group functor defined by polynomial equations with coefficients in $R$. In fact, by choosing a surjection $\pi \colon R[X_1,\ldots,X_n] \to \O(\ag{G})$, we obtain for each commutative $R$-algebra $A$
a natural inclusion
\[
\ag{G}(A) \congrightarrow \hom(\O(\ag{G}), A) \xhookrightarrow{\pi^*} \hom(R[X_1,\ldots,X_n], A) \congrightarrow A^n,
\]
\sloppy which identifies the group $\ag{G}(A)$ with the vanishing set of $\ker(\pi) \subset R[X_1,\ldots,X_n]$ in $A^n$.
For a topological $R$-algebra $A$, we put on $\ag{G}(A)$ the unique weakest topology for which the above (and thus any such) inclusion is continuous with respect to the product topology on $A^n$.

The \term{extension of scalars} of a group scheme $\ag{G}$ over $R$ to some ring extension $S/R$ is the group scheme $\ag{G}_S \colon \cat{Alg}_S \to \cat{Grp}$, $A \mapsto \ag{G}(\res_R(A))$. Here, $\res_R(A)$ denotes $A$ as an $R$-algebra. If the extension map $\sigma \colon R \hookrightarrow S$ is not clear from the context, then we write $\res_\sigma(A)$ instead of $\res_R(A)$ and $\ag{G}_\sigma$ instead of $\ag{G}_S$.

An \term{algebraic group} over a field $K$, or in short a \term{$K$-group}, is now simply a group scheme $\ag{G}$ over $K$.
Its group $\ag{G}(\overline{K})$ of points with values in an algebraic closure $\overline{K}$ of $K$ is then an affine variety in the space $\overline{K}^n$ equipped with the Zariski topology and the group operations are polynomial maps.
The algebraic group $\ag{G}$ is said to be \term{connected} or \term{finite} if this affine variety is connected or finite, respectively.

Let $L/K$ be a finite separable field extension. The \term{restriction of scalars} of an algebraic group $\ag{H}$ over $L$ to $K$ is the functor $\Res_{L/K}\ag{H} \colon \cat{Alg}_K \to \cat{Grp}$, $A \mapsto \ag{H}(A \otimes_K L)$. This is an algebraic group over $K$ (see \cite[p.~57]{Milne-Book}) and the natural isomorphism $K \otimes_K L \congrightarrow L$ induces an isomorphism of topological groups $(\Res_{L/K}\ag{H})(K) \congrightarrow \ag{H}(L)$.

Every algebraic group $\ag{G}$ over a number field $F$ has an \term{integral form}. This is a group scheme $\ag{G}_0$ over the ring of integers $\O_F$ of $F$ together with an $F$-isomorphism $(\ag{G}_0)_F \congrightarrow \ag{G}$.
A subgroup $\Gamma \subset \ag{G}(F)$ is called an \term{arithmetic subgroup} if it is commensurable with the image of $\ag{G}_0(\O_F)$ in $\ag{G}(F)$ for some integral form $\ag{G}_0$ of $\ag{G}$. This notion is independent of the choice of the integral form because the groups of $\O_F$-points of any two integral forms of $\ag{G}$ are commensurable with each other (see \cite[Proposition 4.1]{Platonov}).

By an important theorem of Borel and Harish-Chandra \cite{Borel-Harish-Chandra}, any arithmetic subgroup $\Gamma \subset \ag{H}(\Q)$ of a semisimple algebraic group $\ag{H}$ over $\Q$ is a lattice in $\ag{H}(\R)$. The following definition describes all lattices in the group of real points of an algebraic group over $\R$ that are constructed in this way:

\begin{definition} \label{def:arithmetically-defined-subgroup}
Let $\ag{G}$ be a connected semisimple $\R$-group without $\R$-anisotropic almost $\R$-simple factors.
An irreducible lattice $\Delta \subset \ag{G}(\R)^0$ is \term{arithmetically defined} if there exists
a connected almost $\Q$-simple $\Q$-group $\ag{H}$, an $\R$-epimorphism
\[
\Phi \colon \ag{H}_\R \twoheadrightarrow \ag{G}
\]
so that $(\ker{\Phi})(\R)$ is compact, and an arithmetic subgroup $\Gamma \subset \ag{H}(\Q)$ such that $\Delta$ is commensurable with $\Phi(\ag{H}(\Gamma))$.
\end{definition}

Non-arithmetically defined lattices are known to exist in some real algebraic groups of rank one. For example, in $\ag{SL}_2(\R)$ this follows from the existence of uncountably many non-commensurable hyperbolic surfaces (see \cite[p.~63]{ArithmeticGroupsWhatWhyHow}).

Margulis' celebrated arithmeticity theorem states that in a real algebraic group of higher rank, or similarly in a real Lie group of higher rank, all irreducible lattices are arithmetically defined (see \cite[Theorem~IX.1.11]{Margulis-DiscreteSubgroups}):

\begin{theorem}[Margulis' arithmeticity theorem] \label{thm:margulis-arithmeticity}
Let $\ag{G}$ be a connected semisimple $\R$-group without $\R$-anisotropic almost $\R$-simple factors and
with $\rank_\R(\ag{G}) > 1$. Then every irreducible lattice in $\ag{G}(\R)^0$ is arithmetically defined.
\end{theorem}

\section{Unit groups in quaternion algebras} \label{sec:quaternion-algebras}

The arithmetically defined lattices in $\SL_2(\R)^r$ can be described using quaternion algebras.
Let $K$ be a field of characteristic zero. A \term{quaternion algebra over $K$} is an algebra $D$ over $K$ for which there exists a vector space basis $\{1,i,j,k\}$ and $a,b \in K^\times$ such that
\begin{equation} \label{eq:quaternion-algebra-relations}
i^2 = a,
\quad
j^2 = b,
\quad
k = ij = -ji.
\end{equation}
We then call $\{1,i,j,k\}$ a \term{quaternionic basis} for $D$.
The \term{reduced norm} of an element $x = x_0 + x_1 i + x_2 j + x_3 k \in D$ is
\begin{equation} \label{eq:reduced-norm}
N(x) = x_0^2 - ax_1^2 - bx_2^2 + abx_3^2 \in K.
\end{equation}
The reduced norm of $x \in D$ is preserved by every automorphism of $D$ and is therefore independent of the choice of the quaternionic basis. For example, in the case of the matrix algebra $D = M_2(K)$, we have $N(x) = \det(x)$. An element $x \in D$ is invertible if and only if $N(x) \neq 0$. We write $D^1$ for the group of units of reduced norm one in $D$.

The \term{general linear group} associated to a quaternion algebra $D$ over $K$ is the algebraic group $\ag{GL}_D \colon \cat{Alg}_K \to \cat{Grp}$, $A \mapsto (A \otimes_K D)^\times$. We extend the reduced norm to tensor products $A \otimes_K D$ for $K$-algebras $A$ using \eqref{eq:reduced-norm} and define the \term{special linear group} $\ag{SL}_D \colon \cat{Alg}_K \to \cat{Grp}$, $A \mapsto (A \otimes_K D)^1$.

We now study quaternion algebras over number fields. Let $F$ be a number field and let $D$ be a quaternion algebra over $F$. The analog for $D$ of the ring of integers of a number field is an \term{order} $\Lambda \subset D$. This is a subring $\Lambda$ of $D$ which is a finitely generated $\O_F$-submodule of $D$ and spans $D$ over $F$. For example, the ring $M_2(\O_F)$ is an order of the matrix algebra $M_2(F)$. A quaternion algebra has many orders, but the groups of units of any two of its orders are commensurable with each other (see \cite[Lemma 4.6.9]{GroupRingGroupsI}).

Let $\Lambda \subset D$ be an order. We define the \term{general linear group} $\ag{GL}_\Lambda \colon \cat{Alg}_{\O_F} \to \cat{Grp}$, $A \mapsto (A \otimes_{\O_F} \Lambda)^\times$. This is a group scheme over $\O_F$ and an integral form of the algebraic group $\ag{GL}_D$.
If $\Lambda$ is the $\O_F$-span of a quaternionic basis for $D$, then we again extend the reduced norm using \eqref{eq:reduced-norm} and define the \term{special linear group} $\ag{SL}_\Lambda \colon \cat{Alg}_{\O_F} \to \cat{Grp}$, $A \mapsto (A \otimes_{\O_F} \Lambda)^1$, which is an integral form of $\ag{SL}_D$. It follows that for any order $\Lambda \subset D$, the group of units of reduced norm one $\Lambda^1 := \Lambda \cap D^1$ is an arithmetic subgroup of $\ag{SL}_D(F)$.

For any real embedding $\sigma \colon F \hookrightarrow \R$, the algebra $D \otimes_F \res_\sigma(\R)$ is either isomorphic to $M_2(\R)$ or is a division algebra. In the first case, we say $D$ is \term{split at $\sigma$} and otherwise \term{ramified at $\sigma$}. This leads us to the following definition:

\begin{definition} \label{def:subgroup-derived-from-quaternion-algebra}
A subgroup $\Delta \subset \SL_2(\R)^r$ is said to be \term{derived from a quaternion algebra} if there exists
a totally real number field $F$, a quaternion algebra $D$ over $F$ that is split at exactly $r$ distinct real embeddings $\sigma_1,\ldots,\sigma_r \colon F \hookrightarrow \R$, an isomorphism $\tau_i \colon D \otimes_F \res_{\sigma_i}(\R) \congrightarrow M_2(\R)$ for each $i \in \{1,\ldots,r\}$, and an order $\Lambda \subset D$ such that
\[
\Delta = \bigl\{ \bigl(\tau_1(x), \ldots,\tau_r(x)\bigr) : x \in \Lambda^1 \bigr\}.
\]
\end{definition}

\begin{remark} \label{rem:derived-commensurable}
The maps $\tau_i \colon D \otimes_F \res_\sigma(\R) \congrightarrow M_2(\R)$ in \cref{def:subgroup-derived-from-quaternion-algebra} are not uniquely determined, but any two choices for $\tau_i$ differ only by conjugation with a matrix in $\GL_2(\R)$ by the Skolem--Noether theorem. So up to commensurability and conjugation in $\GL_2(\R)^r$, the subgroup derived from a quaternion algebra $D$ depends only on the isomorphism class of $D$.
\end{remark}

A subgroup derived from a quaternion algebra is readily seen to be an arithmetically defined lattice in $\SL_2(\R)^r = (\ag{SL}_2 \times \ldots \times \ag{SL}_2)(\R)$. In fact, every arithmetically defined lattice in $\SL_2(\R)^r$ comes from a quaternion algebra (see \cite[Theorem 5.44]{ZschummePhD}):

\begin{proposition} \label{thm:classification-arithmetically-defined-lattices}
A subgroup of $\SL_2(\R)^r$ is an arithmetically defined lattice if and only if it is commensurable with a subgroup derived from a quaternion algebra.
\end{proposition}

We write $D = (a,b)_F$ for the quaternion algebra determined by the relations in \eqref{eq:quaternion-algebra-relations}.
The field $E := F(\sqrt{a})$ is then a \term{splitting field for $D$}. By this, we mean that it is a field extension $E/F$ so that $D \otimes_F E \cong M_2(E)$. A quadratic extension $E/F$ is a splitting field for $D$ if and only if for every place $v$ of $F$ where $D$ is ramified, the local completion $E_v/F_v$ is a quadratic extension (see \cite[Theorem~7.3.3 and its proof]{Arithmetics-Reid}). 
Here, $E_v$ is the completion of $E$ at some place of $E$ lying above $v$. Note that for any two places of $E$ lying above $v$, the corresponding completions of $E$ are $F_v$-isomorphic to each other by \cite[Proposition~II.9.1]{Neukirch}, which is why we denote it just by $E_v$. Moreover, for any $a \in F^\times$ for which $F(\sqrt{a})$ is a splitting field for $D$, there exist some $b \in F^\times$ with $D \cong (a,b)_F$ (see \cite[Proposition~1.2.3]{CSAGaloisCohomology}).

Next, we show that the constants $a,b \in F^\times$ defining the isomorphism class of the quaternion algebra $(a,b)_F$ can always be chosen in a certain way. We will need this in our computations later in \cref{sec:geometric-construction}.

\begin{lemma} \label{thm:quaternion-algebra-trick}
Let $F$ be a number field. Then for every quaternion algebra $D$ over $F$, there exist $a,b \in \O_F$ in the ring of integers of $F$  such that $D$ is isomorphic to $(a,b)_F$ and such that for any real embedding $\sigma \colon F \hookrightarrow \R$ at which $D$ is split, we have $\sigma(a) > 0$.
\end{lemma}

\begin{proof}
We use the Grunwald--Wang theorem \cite[p.~29]{Roquette} to construct a splitting field for $D$. By this theorem, there exists a quadratic extension $E/F$ with the following two properties:
\begin{enumerate}
\item For every place $v$ of $F$ where $D$ is ramified, $E_v/F_v$ is a quadratic extension.
\item For every real place $v$ of $F$ where $D$ is split, $E_v/F_v$ is a trivial extension.
\end{enumerate}
Then $E/F$ is a splitting field for $D$ by the criterion stated above.
Moreover, since $E/F$ is a quadratic extension, we have $E = F(\sqrt{a})$ for some $a \in F^\times$.
Let $\sigma \colon F \hookrightarrow \R$ be a real embedding at which $D$ is split. By the second property, we have $E_v \cong \R$ for the place $v$ of $F$  corresponding to $\sigma$. Hence the image of the extension $\widetilde{\sigma} \colon E \hookrightarrow \C$ of $\sigma$ given by $\widetilde{\sigma}(x + y\sqrt{a}) = \sigma(x) + \sigma(y)\sqrt{\sigma(a)}$ for $x,y \in F$ stays in $\R$, and so we must have $\sigma(a) > 0$.

Since $E/F$ is a splitting field for $D$, we find $b \in F^\times$ with $D \cong (a,b)_F$. Finally, we can also achieve that $a,b \in \O_F$ because of $(a,b)_F \cong (c^2a,c^2b)_F$ for all $c \in F^\times$ by \cite[p.~78]{Arithmetics-Reid}.
\end{proof}

\section{Adeles and congruence subgroups} \label{sec:adeles-and-congruence-subgroups}

We will use adeles and congruence subgroups to construct subgroups of finite index in arithmetic groups.
Let $F$ be a number field with ring of integers $\O_F$. We denote by $\AA{F}$ the ring of finite adeles of $F$ and by $\OO{F}$ the ring of integral finite adeles of $F$ (see \cite[pp.~10--13]{Platonov}). We consider $F$ as a subring of $\AA{F}$ by the diagonal embedding $F \hookrightarrow \AA{F}$, and similarly $\O_F$ as a subring of $\OO{F}$ by the diagonal embedding $\O_F \hookrightarrow \OO{F}$. This turns $\AA{F}$ and $\OO{F}$ into topological algebras over $F$ and $\O_F$, respectively.

Let now $\ag{G}$ be a group scheme over $\O_F$. For a subgroup $U \subset \ag{G}(\OO{F})$ and a nonzero ideal $\a \subset \O_F$, we write
\[
U(\a) := \ker\bigl(U \to \ag{G}(\OO{F}/\a\OO{F})\bigr),
\]
where $\a\OO{F}$ is the ideal in $\OO{F}$ generated by $\a$. We have (see \cite[Proposition 4.39]{ZschummePhD}):

\begin{proposition} \label{thm:adelic-points-basis}
Let $\ag{G}$ be a group scheme over the ring of integers $\O_F$ of a number field $F$. Then the family of groups $\ag{G}(\OO{F})(\a)$ for all nonzero ideals $\a \subset \O_F$ is a basis of open neighborhoods of the identity in both of the groups $\ag{G}(\OO{F})$ and $\ag{G}(\AA{F})$.
\end{proposition}

For a subgroup of the integral points $\Gamma \subset \ag{G}(\O_F)$, the group $\Gamma(\a)$ is the kernel of the map $\Gamma \to \ag{G}(\O_F/\a\O_F)$ and has finite index in $\Gamma$. We call $\Gamma(\a)$ the \term{principal congruence subgroup of $\Gamma$ of level $\a$}.
More generally, any subgroup of $\Gamma$ that contains a principal congruence subgroup of $\Gamma$ has finite index in $\Gamma$ and is called a \term{congruence subgroup of $\Gamma$}.

In some group schemes $\ag{G}$, the group $\ag{G}(\O_F)$ has finite index subgroups which are not congruence subgroups. Examples of such subgroups in $\ag{SL}_2(\Z)$ were already known to Fricke and Klein in the $19$th century (see \cite[p.~299]{CongruenceSubgroupProblem}). The group scheme $\ag{G}$ is said to have the \term{congruence subgroup property} if every subgroup of finite index in $\ag{G}(\O_F)$ is a congruence subgroup. Chevalley \cite{Chevalley} proved in 1951 that $\ag{GL}_1$ has this property:

\begin{theorem}[Chevalley] \label{thm:chevalley}
Let $F$ be a number field. Then for every subgroup of finite index $\Gamma \subset \O_F^\times$, there exists a nonzero ideal $\a \subset \O_F$ such that $\O_F^\times(\a) \subset \Gamma$.
\end{theorem}

\section{Polar regular elements and flats} \label{sec:flats}

We will use polar regular elements, as introduced by Mostow in \cite[p.~12]{StrongRigidity}, to algebraically describe the maximal flat subspaces of a symmetric space.

Let $G$ be a connected linear semisimple Lie group. By \cite[p.~431]{Helgason}, each $g \in G$ has a unique decomposition
\begin{equation} \label{eq:real-jordan-decomposition}
g = g_u g_h g_e
\end{equation}
with $g_u,g_h,g_e \in G$ such that $g_u$,$g_h$, and $g_e$ correspond in one (and thus any) embedding $G \hookrightarrow \GL_n(\R)$ to a unipotent, a hyperbolic, and an elliptic matrix, respectively, and such that they all commute with each other.  We call a matrix in $\GL_n(\R)$ \term{semisimple} if it is diagonalizable over $\C$. A semisimple matrix is called \term{hyperbolic} if all its eigenvalues are real and positive, and it is called \term{elliptic} if all its eigenvalues have absolute norm one. We call \eqref{eq:real-jordan-decomposition} the \term{real Jordan decomposition of $g$}.

\begin{definition} \label{def:polar-regular}
An element $g \in G$ is \term{polar regular} if for all $g' \in G$, we have $$\dim(C_G(g_h)) \leq \dim(C_G(g'_h)),$$
where $C_G(g_h)$ denotes the centralizer of $g_h$ in $G$.
\end{definition}

Let $X_G$ be the symmetric space associated to $G$. A \term{flat} in $X_G$ is a connected totally geodesic submanifold of $X_G$ whose curvature tensor vanishes. A flat is called \term{maximal} if it is of maximal dimension among all flats in $X_G$. For a flat $A \subset X_G$, we denote its stabilizer subgroup in $G$ by
\[
G_A := \{ g \in G : g \cdot A = A\}.
\]
By \cite[Lemma~5.2]{StrongRigidity}, we have the following relationship between polar regular elements in $G$ and maximal flats in $X_G$:

\begin{proposition} \label{thm:polar-regular-mostow}
Let $G$ be a connected linear semisimple Lie group and let $g \in G$ be polar regular. Then there exists a unique maximal flat $A \subset X_G$ in the symmetric space associated to $G$ such that $g \cdot A = A$. Moreover, the centralizer $C_G(g)$ is a subgroup of $G_A$ and acts transitively on $A$.
\end{proposition}

\begin{example} \label{ex:flats-product-hyperbolic-space}
Let $G = \SL_2(\R)^r$. Then $X_G = (\H^2)^r$. The maximal flats in $(\H^2)^r$ are the products of geodesic lines in $\H^2$. An element $g = (g_1,\ldots,g_r) \in \SL_2(\R)^r$ is polar regular if and only if each $g_i$ has two distinct real eigenvalues (see \cite[Lemmas 3.19 and 3.29]{ZschummePhD}).
\end{example}

Let $\Gamma \subset G$ be a lattice. We say that a flat $A \subset X_G$ is \term{$\Gamma$-compact} if the quotient $\Gamma_A \backslash A$ is compact. Then the image of $A$ in $\Gamma \backslash X_G$ is also compact.
By \cite[Lemma~8.3']{StrongRigidity}, the set of $\Gamma$-compact maximal flats is dense in the space of all maximal flats:

\begin{theorem}[Density of $\Gamma$-compact flats] \label{thm:density}
Let $G$ be a connected linear semisimple Lie group. Let $\Gamma \subset G$ be a lattice and let $A \subset X_G$ be a maximal flat.
Then for every open neighborhood of the identity $U \subset G$, there exists some $u \in U$ such that $u \cdot A$ is a $\Gamma$-compact maximal flat in $X_G$ that is stabilized by a polar regular element of $\Gamma$.
\end{theorem}

\section{Geometric cycles} \label{sec:geometric-cycles}

Let $G$ be a linear semisimple Lie group and let $X_G$ be the symmetric space associated to $G$.
For a closed subgroup $H \subset G$, we can always find a maximal compact subgroup $K \subset G$ such that $K_H := K \cap H$ is a maximal compact subgroup of $H$. We write $X_H := H/K_H$. Then $X_H$ is diffeomorphic to a Euclidean space and the inclusion $H \hookrightarrow G$ induces a closed embedding
\[
j_H \colon X_H \hookrightarrow X_G
\]
whose image is a totally geodesic submanifold of $X_G$ (see \cite[p.~213]{Schwermer}).

Consider now a torsion-free lattice $\Gamma \subset G$ and the corresponding locally symmetric space $\Gamma \backslash X_G$. The map $j_H$ descends to an immersion into the locally symmetric space $\Gamma \backslash X_G$, but this map will in general not be an embedding. In the arithmetic setting, we can always obtain an embedding by passing to a subgroup of finite index (see \cite[Theorem D]{Schwermer}):

\begin{theorem} \label{thm:geometric-cycles}
Let $\ag{G}$ be a connected semisimple $\Q$-group and let $\ag{H}$ be a connected reductive $\Q$-subgroup of $\ag{G}$. Then any arithmetic subgroup $\Gamma \subset \ag{G}(\Q)$ has a torsion-free subgroup of finite index $\Gamma_0 \subset \Gamma$ such that for any subgroup of finite index $\Gamma' \subset \Gamma_0$, the map
\[
j_{H \mid \Gamma'} \colon (\Gamma' \cap H) \backslash X_H \to \Gamma' \backslash X_G
\]
induced by the above map $j_H \colon X_H \hookrightarrow X_G$ for $G = \ag{G}(\R)$ and $H = \ag{H}(\R)$ is a closed embedding and its image is an orientable totally geodesic submanifold.
\end{theorem}

In the situation of the above theorem, we call the image of $j_{H|\Gamma'}$ in $\Gamma' \backslash X_G$ a \term{geometric cycle}. An effective strategy to show that the fundamental class of a geometric cycle
is nontrivial in the homology of $\Gamma' \backslash X_G$ is to find another
geometric cycle in $\Gamma' \backslash X_G$ such that their intersection product is nontrivial.
We will do this in the next section.

\section{Construction of flat submanifolds} \label{sec:geometric-construction}

In this section, we explain our construction of the compact oriented flat totally geodesic submanifolds with linearly independent homology classes in an arithmetic Riemannian manifold covered by $(\H^2)^r$. This will finish our proof of \cref{thm:main}.
We will proceed as follows:
\begin{itemize}
\item In \cref{subsec:configuration}, we construct for an arithmetic quotient $M$ of $(\H^2)^r$ and any given number $n \in \N$ two families $(A_i)_{1 \leq i \leq n}$ and $(B_j)_{1 \leq j \leq n}$ of maximal flats in $(\H^2)^r$ so that $A_i$ and $B_j$ intersect if and only if $i \leq j$ and so that the images of the $A_i$ in $M$ are compact.
\item Next, in \cref{subsec:controlling}, we study the intersections of the images of the flats $A_i$ and $B_j$ in finite covering spaces $M'$ of $M$. Let us denote these images by $\mathcal{F}_i$ and $\mathcal{G}_j$, respectively. We will show that by suitably choosing $M'$, one can achieve that $\mathcal{F}_i$ and $\mathcal{G}_j$ are embedded submanifolds of $M'$ and that one can control their intersection numbers.
\item Finally, in \cref{subsec:finishing}, we will use our knowledge about the intersection numbers of the submanifolds $\mathcal{F}_i$ and $\mathcal{G}_j$ of $M'$ combined with methods from de Rham cohomology to show that the fundamental classes $[\mathcal{F}_1], \ldots, [\mathcal{F}_n]$ are linearly independent in $H_r(M';\Z)$.
\end{itemize}
As we have seen in \cref{sec:algebraic-groups,sec:quaternion-algebras}, it suffices to consider for our task a Riemannian manifold $M = \Delta \backslash (\H^2)^r$, where $\Delta \subset \SL_2(\R)^r$ is a subgroup derived from a quaternion algebra.

Throughout this section, we therefore fix the following notation and assumptions:
Let $F$ be a totally real number field. We denote by $\{\sigma_1,\ldots,\sigma_d\}$ the set of all distinct embeddings $F \hookrightarrow \R$. Further, $D = (a,b)_F$ is a quaternion algebra over $F$ that is split at the first $r$ embeddings and ramified at the remaining ones.
We assume that $\Lambda \subset D$ is an order and $\Gamma \subset \Lambda^1$ is a torsion-free subgroup of finite index.
For each $i \in \{1,\ldots,r\}$, we fix an isomorphism $\tau_i \colon D \otimes_F \res_{\sigma_i}(\R) \congrightarrow M_2(\R)$. We assume that $\Delta$ is given by
\[
\Delta = \big\{ (\tau_1(x), \ldots, \tau_r(x)) : x \in \Gamma \big\} \subset \SL_2(\R)^r.
\]
We assume that $F \subset \R$ and $\sigma_1$ is the identity embedding.
Further, we assume that $a,b \in \O_F$ and $\sigma_1(a), \ldots, \sigma_r(a) > 0$. This is justified by \cref{thm:quaternion-algebra-trick}.
We also assume that the order $\Lambda$ is the $\O_F$-span of a quaternionic basis $\{1,i,j,k\}$ for $D$ and that with respect to this basis, $\tau_1$ is given by
\[
\tau_1(x+yi+zj+wk)
= \begin{pmatrix}x + y\sqrt{a} & z + w\sqrt{a} \\ b(z - w\sqrt{a}) & x-y\sqrt{a}\end{pmatrix}.
\]
This is justified by \cref{rem:derived-commensurable}. In particular, we have $\tau_1(D) \subset M_2(F(\sqrt{a}))$.
In this setting, we have an action of $D^1$ on $(\H^2)^r$ by
\[
g \cdot (z_1,\ldots,z_r) := (\tau_1(g)\cdot z_1, \ldots, \tau_r(g) \cdot z_r),
\]
where $\tau_i(g) \cdot z_i$ is the action of $\SL_2(\R)$ on $\H^2$ by fractional linear transformations.
Note that we have $M = \Gamma \backslash (\H^2)^r$.
We extend the above action of $D^1$ to an action of the group $D^\times$ on $(\H^2)^r$ through the maps $\tau_i$ by defining the action of $\GL_2(\R)$ on $\H^2$ as follows:
\[
\begin{pmatrix}
a_{11} & a_{12} \\
a_{21} & a_{22}
\end{pmatrix}
\cdot
z :=
\begin{cases}
(a_{11}z+a_{12})(a_{21}z+a_{22})^{-1}, & \text{if } a_{11}a_{22}-a_{12}a_{21} > 0, \\[8pt]
\overline{(a_{11}z+a_{12})(a_{21}z+a_{22})^{-1}}, & \text{otherwise. }
\end{cases}
\]
We write $(D^\times)_A := \{ x \in D^\times : x \cdot A = A \}$ for the stabilizer group of a flat $A \subset (\H^2)^r$.
One can check that for our action of $D^\times$ on $(\H^2)^r$, the following analog of \cref{thm:polar-regular-mostow} holds (see  \cite[Proposition 3.31]{ZschummePhD} for a proof):

\begin{proposition} \label{thm:polar-regular-quaternion-algebra}
Let $x \in D^\times$ be such that $\tau_1(x),\ldots,\tau_r(x) \in \GL_2(\R)$ each have two distinct real eigenvalues.
Then there exists a unique maximal flat $A \subset (\H^2)^r$ with $x \cdot A = A$.
Moreover, the centralizer $C_{D^\times}(x)$ is a subgroup of finite index in $(D^\times)_A$ and acts by orientation-preserving isometries on $A$.
\end{proposition}

\subsection{Building a configuration of flats} \label{subsec:configuration}

\noindent We now start with the construction of the families of flats $(A_i)_{1 \leq i \leq n}$ and $(B_j)_{1 \leq j \leq n}$ in the symmetric space $(\H^2)^r$. The following lemma shows that two generic geodesic lines in $\H^2$ can be slightly perturbed without changing the way they intersect. Here, we denote by $\partial_\infty \H^2$ the boundary at infinity of the hyperbolic plane (see \cite[p.~27]{Eberlein}).

\begin{lemma}[Perturbation Lemma] \label{thm:perturbation-lemma}
Let $L_1$ and $L_2$ be two geodesic lines in $\H^2$ whose four endpoints in $\partial_\infty \H^2$ are pairwise distinct.
Then $L_1$ and $L_2$ are either disjoint or intersect transversally in a single point. Moreover, there exists an open neighborhood of the identity $U \subset \SL_2(\R)$ such that for all $u,v \in U$, the geodesic lines $u \cdot L_1$ and $v \cdot L_2$ also have pairwise distinct endpoints in $\partial_\infty \H^2$ and intersect in the same way as $L_1$ and $L_2$.
\end{lemma}

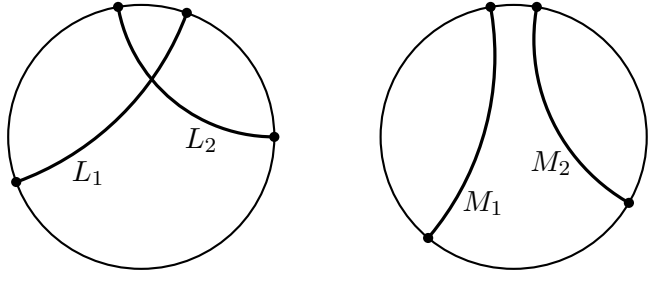
\begin{figure}
\centering
\begin{minipage}[b]{.375\textwidth}
\centering
\begin{tikzpicture}[scale=1.75]
\begin{scope}
\draw (0,0) circle (1);
\clip (0,0) circle (1cm);
\hgline[very thick]{200}{70};
\hgline[very thick]{100}{0};
\end{scope}
\node[vertex] at (200:1) {};
\node[vertex] at (70:1) {};
\node[vertex] at (100:1) {};
\node[vertex] at (0:1cm) {};
\node at (-0.4,-0.275) {$L_1$};
\node at (0.45,-0.025) {$L_2$};
\end{tikzpicture}
\end{minipage}
\begin{minipage}[b]{.375\textwidth}
\centering
\begin{tikzpicture}[scale=1.75]
\begin{scope}
\draw (0,0) circle (1);
\clip (0,0) circle (1cm);
\hgline[very thick]{230}{100};
\hgline[very thick]{80}{-30};
\end{scope}
\node[vertex] at (230:1cm) {};
\node[vertex] at (100:1cm) {};
\node[vertex] at (80:1cm) {};
\node[vertex] at (-30:1cm) {};

\node at (0.28,-0.2) {$M_2$};
\node at (-0.23,-0.5) {$M_1$};
\end{tikzpicture}
\end{minipage}
\caption{Two pairs of geodesic lines in $\H^2$ and their endpoints in the boundary at infinity $\partial_\infty \H^2$.}	
\label{fig:geodesics-endpoints}
\end{figure}

\begin{proof}
Note that $L_1$ and $L_2$ intersect if and only if their endpoints in $\partial_\infty \H^2$ are linked, in which case their intersection consists of a single point and is transverse (see \cref{fig:geodesics-endpoints}). We denote the four endpoints of $L_1$ and $L_2$ by $v_1,v_2,v_3,v_4 \in \partial_\infty\H^2$.
The geodesic compactification $\overline{\H^2} := \H^2 \sqcup \partial_\infty\H^2$ as defined in \cite[pp.~28--30]{Eberlein} is Hausdorff, and so we can find pairwise disjoint open neighborhoods $V_i \subset \overline{\H^2}$ of the points $v_i$.
The group $\SL_2(\R)$ acts continuously on $\overline{\H^2}$, hence the maps $\phi_i \colon \SL_2(\R) \to \overline{\H^2}$ given by $\phi_i(g) := g \cdot v_i$ are continuous. So $U := \bigcap_{i=1}^4 \phi_i^{-1}(V_i)$ is an open neighborhood of the identity in $\SL_2(\R)$. By the construction of $U$, the endpoints of two geodesic lines $u \cdot L_1$ and $v \cdot L_2$ with $u,v \in U$ in $\partial_\infty\H^2$ are linked if and only if whose of $L_1$ and $L_2$ are linked. So the statement of the lemma follows.
\end{proof}

We can now describe the construction of the families of flats $(A_i)_{1 \leq i \leq n}$ and $(B_j)_{1 \leq j \leq n}$ adapted to the group $\Gamma$:

\begin{proposition} \label{thm:configuration}
For any $n \in \N$, there exist maximal flats $A_1,\ldots,A_n$ and $B_1,\ldots,B_n$ in $(\H^2)^r$ so that the following conditions are satisfied:
\begin{enumerate}
\item The flats $A_i$ and $B_j$ intersect transversally in a single point if $i \leq j$, and they are disjoint if $i > j$.
\item Each $A_i$ is $\Gamma$-compact.
\item Each $A_i$ is stabilized by an element $\alpha_i \in D^1$ such that $\tau_1(\alpha_i),\ldots,\tau_r(\alpha_i)$ each have two distinct real eigenvalues.
\item Each $B_j$ is stabilized by an element $\beta_j \in D^\times$ such that $\tau_1(\beta_j),\ldots,\tau_r(\beta_j)$ each have two distinct real eigenvalues and $\tau_1(\beta_j)$ is diagonalizable over $F(\sqrt{a})$.
\end{enumerate}
\end{proposition}

\begin{proof}
We start with the first condition. For this, we choose geodesic lines $L_1,\ldots,L_n$ and $M_1,\ldots,M_n$ in $\H^2$ with pairwise distinct endpoints in $\partial_\infty\H^2$ such that for each $i,j \in \{1,\ldots,n\}$, we have that $L_i$ and $M_j$ intersect if and only if $i \leq j$ (see \cref{fig:geodesics-pattern} for an example in the case $n=3$). Set $A_i := L_i \times \cdots \times L_i$ and $B_j := M_j \times \cdots \times M_j$. The first condition is now satisfied.

By \cref{thm:perturbation-lemma}, there exists an open neighborhood of the identity $U \subset \SL_2(\R)$ so that we may perturb each $L_i$ and $M_j$ by isometries in $U$ without invalidating the first condition. The product $U^r$ is an open neighborhood of the identity in $\SL_2(\R)^r$. So for each $i \in \{1,\ldots,n\}$, there exists by \cref{thm:density} an element $u_i \in U^r$ such that $u_i \cdot A_i$ is a $\Gamma$-compact flat that is stabilized by an element $\alpha_i \in D^1$ for which $\tau_1(\alpha_i),\ldots,\tau_r(\alpha_i)$ each have two distinct real eigenvalues. We now replace each $A_i$ by $u_i \cdot A_i$ and the first three conditions are satisfied.

\begin{figure}[t]
\centering
\begin{tikzpicture}[scale=2.3]
\draw (0,0) circle (1);
\clip (0,0) circle (1);
\hgline[very thick]{170}{110} node[blue,xshift=-8,yshift=-10] {};
\hgline[very thick]{190}{90} node[blue,xshift=-10,yshift=-13] {};
\hgline[very thick]{215}{70} node[blue,xshift=-12,yshift=-18] {};
\hgline[very thick]{120}{-25} node[red,xshift=5,yshift=-5] {};
\hgline[very thick]{100}{0}  node[red,xshift=5,yshift=-5] {};
\hgline[very thick]{80}{20}  node[red,xshift=5,yshift=-5] {};
\node at (-0.737,0.065) {$L_3$};
\node at (-0.6605,-0.2137) {$L_2$};
\node at (-0.5566,-0.5306) {$L_1$};
\node at (0.7709,0.1904) {$M_1$};
\node at (0.7273,-0.0991) {$M_2$};
\node at (0.6017,-0.405) {$M_3$};
\end{tikzpicture}
\caption{A pattern of geodesics lines in $\H^2$.}
\label{fig:geodesics-pattern}
\end{figure}
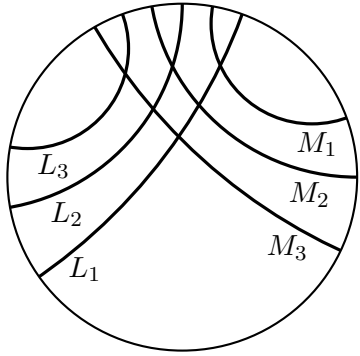

For the last condition, we choose an element $x_0 \in D^\times$ with $(x_0)^2 = a$ and $x_0 \notin F$.
Then for each $k \in \{1,\ldots,r\}$, we have $(\tau_k(x_0))^2 = \sigma_k(a)I_2$ and $\tau_k(x_0) \notin \R \cdot I_2$, where $I_2 \in M_2(\R)$ is the identity matrix. Recall that $\sigma_k(a) > 0$ by assumption. So the minimal polynomial of $\tau_k(x_0)$ over $\R$ is $\bigl(X + \sqrt{\sigma_k\smash[b]{(a)}}\bigr)\bigl(X - \sqrt{\sigma_k\smash[b]{(a)}}\bigr) \in \R[X]$.
Hence, each $\tau_k(x_0)$ has two distinct real eigenvalues and $\tau_1(x_0)$ is diagonalizable over $F(\sqrt{a})$.
So by \cref{thm:polar-regular-quaternion-algebra}, there exists a unique maximal flat $B_0 \subset (\H^2)^r$ that is stabilized by $x_0$.
The group $\SL_2(\R)^r$ acts transitively on the set of all maximal flats in $(\H^2)^r$, and so for each $j \in \{1,\ldots,n\}$, we can find some $T_j \in \SL_2(\R)^r$ with $B_j = T_j \cdot B_0$.

By the real approximation theorem \cite[Theorem~25.70]{Milne-Book}, the image of $D^1$ in $\SL_2(\R)^r$ under the maps $\tau_1,\ldots,\tau_r$ is dense. Since the subsets $UT_j \subset \SL_2(\R)^r$ are open, there exist $x_j \in D^1$ and $v_j \in U$ with
$(\tau_1(x_j),\ldots,\tau_r(x_j)) = v_j T_j$.
Next, from $v_j \cdot B_j = v_jT_j \cdot (T_j^{-1} \cdot B_j) = x_j \cdot B_0$, we conclude that
\[
x_jx_0x_j^{-1} \cdot (v_j \cdot B_j) = x_j x_0 \cdot B_0 = x_j \cdot B_0 = v_j \cdot B_j.
\]
This shows that $v_j \cdot B_j \subset (\H^2)^r$ is stabilized by $\beta_j := x_j x_0 x_j^{-1} \in D^\times$. We now replace each $B_j$ by $v_j \cdot B_j$ and then all four conditions are satisfied.
\end{proof}

\subsection{Controlling the intersections in the quotient} \label{subsec:controlling}

\noindent Our next goal is to find a finite covering space of $M = \Gamma \backslash (\H^2)^r$ in which the images of the flats $A_i$ and $B_j$ from \cref{thm:configuration} are embedded submanifolds and to control the intersections of these submanifolds.

To ease notation, we fix throughout this subsection two maximal flats $A, B \subset (\H^2)^r$ that are either disjoint or intersect transversally in a single point and we fix $\alpha \in D^1$ and $\beta \in D^\times$ that stabilize $A$ and $B$, respectively. We assume that $\tau_1(\alpha),\ldots,\tau_r(\alpha)$ and $\tau_1(\beta),\ldots,\tau_r(\beta)$ each have two distinct real eigenvalues and that $\tau_1(\beta)$ is diagonalizable over $F(\sqrt{a})$.
We also assume that $A$ is $\Gamma$-compact.

\begin{proposition} \label{thm:centralizer}
There exists a subgroup of finite index $\Gamma_{\cent} \subset \Gamma$
such that every element of $\Gamma_\cent$ which stabilizes $A$ commutes with $\alpha$, and every element of $\Gamma_\cent$ which stabilizes $B$ commutes with $\beta$.
\end{proposition}

\begin{proof}
By \cref{thm:polar-regular-quaternion-algebra}, the centralizer $C_{D^\times}(\alpha)$ is a subgroup of finite index in $(D^\times)_A$. So there exist $y_1,\ldots,y_m \in (D^\times)_A$ with
\[
(D^\times)_A = C_{D^\times}(\alpha) \sqcup C_{D^\times}(\alpha)y_1 \sqcup \ldots \sqcup C_{D^\times}(\alpha)y_m.
\]
We now use the rings $\AA{F}$ and $\OO{F}$ from \cref{sec:adeles-and-congruence-subgroups}.
For each $i \in \{1,\ldots,m\}$, the set $C_{\ag{GL}_D(\AA{F})}(\alpha)y_i \subset \ag{GL}_D(\AA{F})$ is a closed subset that does not contain the identity. Hence, by \cref{thm:adelic-points-basis}, we find a nonzero ideal $\a \subset \O_F$ such that for all $i \in \{1,\ldots,m\}$, we have $\ag{GL}_{\Lambda}(\OO{F})(\a) \cap C_{\ag{GL}_D(\AA{F})}(\alpha)y_i = \emptyset$.
Consequently, we have
\[
\Gamma(\a)_A \subset (D^\times)_A \cap \ag{GL}_\Lambda(\OO{F})(\a) \subset C_{D^\times}(\alpha).
\]
Similarly, we find a nonzero ideal $\b \subset \O_F$ with $\Gamma(\b)_B \subset C_{D^\times}(\beta)$.
Then $\Gamma_\cent := \Gamma(\a) \cap \Gamma(\b)$ is of finite index in $\Gamma$, and the proof is complete.
\end{proof}

We can now show that the images of $A$ and $B$ in some finite covering space of $\Gamma \backslash (\H^2)^r$ are embedded submanifolds:

\begin{proposition} \label{thm:embedding}
There exists a subgroup of finite index $\Gamma_{\emb} \subset \Gamma_{\cent}$ so that for every subgroup of finite index $\Gamma' \subset \Gamma_{\emb}$, the natural maps
\[
\Gamma'_A \backslash A \to \Gamma' \backslash (\H^2)^r
\quad \text{and} \quad
\Gamma'_B \backslash B \to \Gamma' \backslash (\H^2)^r
\]
are closed embeddings whose images are orientable flat totally geodesic $r$-dimensional submanifolds of $\Gamma' \backslash (\H^2)^r$.
\end{proposition}

\begin{proof}
The algebraic group $\ag{SL}_D$ is connected and semisimple.
Because of $\alpha \in \ag{SL}_D(F)$, there exists by \cite[pp.~33--35]{Milne-Book} a unique smallest $F$-subgroup $C_{\ag{SL}_D}(\alpha)$ of $\ag{SL}_D$ such that for all fields $K$ with $F \subset K$, we have
\[
\bigl(C_{\ag{SL}_D}(\alpha)\bigr)(K) = C_{\ag{SL}_D(K)}(\alpha).
\]
Since $\tau_1(\alpha)$ is a diagonalizable matrix,
we see that $C_{\ag{SL}_D}(\alpha)$ becomes isomorphic to $\ag{GL}_1$ over an algebraic closure of $F$,
so $C_{\ag{SL}_D}(\alpha)$ is connected and reductive.

Let $\ag{G} := \Res_{F/\Q}(\ag{SL}_D)$ and $\ag{H} := \Res_{F/\Q}\bigl(C_{\ag{SL}_D}(\alpha)\bigr)$. Then $\ag{G}$ is a connected semisimple $\Q$-group and $\ag{H}$ is a connected reductive $\Q$-subgroup of $\ag{G}$.
For each $i \in \{r+1,\ldots,d\}$, we choose an isomorphism $\rho_i \colon D \otimes_F \res_{\sigma_i}(\R) \congrightarrow \Hamilton$, where $\Hamilton$ is the real Hamilton quaternion algebra. The isomorphisms $\tau_1,\ldots,\tau_r$ and $\rho_{r+1},\ldots,\rho_d$ induce an isomorphism
\begin{align*}
G := \ag{G}(\R) \congrightarrow \SL_2(\R)^r \times (\Hamilton^1)^{d-r},
\end{align*}
which maps the group $H := \ag{H}(\R)$ to $\prod_{i=1}^r C_{\SL_2(\R)}(\tau_i(\alpha)) \times \prod_{i=r+1}^{d} C_{\Hamilton^1}(\rho_i(\alpha))$. Let $K \subset G$ be the preimage of $\SO(2)^r \times (\Hamilton^1)^{d-r}$ under this isomorphism. Then $K$ is a maximal compact subgroup of $G$ and $K_H := K \cap H$ is a maximal compact subgroup of $H$.
Consider now the quotient spaces $X_G := G/K$ and $X_H := H/K_H$ and the embedding $j_H \colon X_H \hookrightarrow X_G$ induced by the inclusion $H \hookrightarrow G$. Fix a point $x_0 \in A$. Then the diffeomorphism
\[
X_G \congrightarrow (\H^2)^r,
\quad
gK \mapsto g \cdot x_0
\]
maps $j_H(X_H)$ onto the flat $A$.
Note that $\ag{H}(\Q) = C_{D^1}(\alpha)$. So by \cref{thm:centralizer,thm:polar-regular-quaternion-algebra}, we have $\Gamma'_A = \Gamma' \cap \ag{H}(\Q)$ for every subgroup of finite index $\Gamma' \subset \Gamma_\cent$.
Hence, by \cref{thm:geometric-cycles}, there exists a subgroup of finite index $\Gamma_0 \subset \Gamma_\cent$ such that for all subgroups of finite index $\Gamma' \subset \Gamma_0$, the map $\Gamma'_A \backslash A \to \Gamma' \backslash (\H^2)^r$ is a closed embedding whose image is an orientable flat totally geodesic submanifold.

Similarly, we obtain a subgroup of finite index $\Gamma_1 \subset \Gamma_\cent$ such that for every subgroup of finite index $\Gamma' \subset \Gamma_1$, the map $\Gamma'_B \backslash B \to \Gamma' \backslash (\H^2)^r$ is a closed embedding whose image is an orientable flat totally geodesic submanifold.
We set $\Gamma_\emb := \Gamma_0 \cap \Gamma_1$ and the proof is complete.
\end{proof}

\begin{lemma} \label{thm:technical-disjoint-union}
Let $\Gamma' \subset \Gamma_{\emb}$ be a subgroup of finite index. Then $\Gamma'A$ is a disjoint union of copies of $A$, that is, for any $\gamma \in \Gamma'$, we have $\gamma A = A$ or $\gamma A \cap A = \emptyset$. Similarly, $\Gamma'B$ is a disjoint union of copies of $B$.
\end{lemma}

\begin{proof}
Let $\gamma \in \Gamma'$ and assume that $\gamma A \cap A \neq \emptyset$. Then there exist $x_1,x_2 \in A$ with $x_2 = \gamma x_1$, and so we have $\Gamma' x_1 = \Gamma' x_2$. Since the map $\Gamma'_A \backslash A \to \Gamma' \backslash (\H^2)^r$ is injective by \cref{thm:embedding}, it follows that $\Gamma'_{A} x_1 = \Gamma'_{A} x_2$. So there exists some $\delta \in \Gamma'_A$ with $x_1 = \delta x_2$. Hence, we have $\delta\gamma x_1 = x_1$, and because $\Gamma'$ is torsion-free, this implies that $\gamma = \delta^{-1}$. So we have $\gamma \in \Gamma'_A$, or, in other words, $\gamma A = A$.
The statement for $\Gamma'B$ can be proven analogously.
\end{proof}	

\begin{remark} \label{rem:induced-orientations}
For any subgroup of finite index $\Gamma' \subset \Gamma_{\emb}$, the images of the flats $A$ and $B$ in $\Gamma' \backslash (\H^2)^r$ can be oriented as follows: We choose orientations $A^+$ on $A$ and $B^+$ on $B$. Then we define $\Gamma'$-invariant orientations on $\Gamma'A$ and $\Gamma'B$ by $(\gamma A)^+ := \gamma A^+$ and $(\gamma B)^+ := \gamma B^+$ for any $\gamma \in \Gamma'$.
Note that by \cref{thm:embedding}, the maps $\Gamma'A \to \Gamma'_A \backslash A$ and $\Gamma'B \to \Gamma'_B \backslash B$ are covering maps and their images are diffeomorphic to the images of $A$ and $B$ in $\Gamma' \backslash (\H^2)^r$, respectively.
Since $\Gamma_{\emb} \subset \Gamma_\cent$, we have by \cref{thm:centralizer,thm:polar-regular-quaternion-algebra} that $\Gamma'_A$ and $\Gamma'_B$ act by orientation-preserving isometries on $A$ and $B$, respectively. So we get induced orientations on the images of $A$ and $B$ in $\Gamma' \backslash (\H^2)^r$.
\end{remark}

Our next task is to find a finite covering space of the locally symmetric space $\Gamma \backslash (\H^2)^r$ in which we can control the intersection of the images of $A$ and $B$. We start with some technical results:

\begin{lemma} \label{thm:technical-double-cosets}
Let $\Gamma' \subset \Gamma_{\emb}$ be a subgroup of finite index. Then for any $\gamma_1,\gamma_2 \in \Gamma'$, we have
$\Gamma'_{B}\gamma_1 A = \Gamma'_{B}\gamma_2 A$ if and only if there exists some $\delta \in \Gamma'_{B}$ with $\gamma_1 A = \delta \gamma_2 A$.
\end{lemma}

\begin{proof}
If there exists some $\delta \in \Gamma'_{B}$ with $\gamma_1 A = \delta \gamma_2 A$, then we also have $\Gamma'_{B}\gamma_1 A = \Gamma'_{B}\gamma_2 A$. Conversely, if $\Gamma'_B\gamma_1 A = \Gamma'_{B}\gamma_2 A$, then $\gamma_1 A \cap \delta \gamma_2 A \neq \emptyset$ for some $\delta \in \Gamma'_{B}$, and so $\gamma_1 A = \delta \gamma_2 A$ by \cref{thm:technical-disjoint-union}.
\end{proof}

\begin{proposition} \label{thm:finiteness}
For every subgroup of finite index $\Gamma' \subset \Gamma_\emb$, we have
\[
\# \bigl\{ \Gamma'_B\gamma A : \gamma \in \Gamma' \text{ with } \gamma A \cap B \neq \emptyset \bigr\} < \infty.
\]
\end{proposition}

\begin{proof}
Let $\pi \colon (\H^2)^r \to \Gamma' \backslash (\H^2)^r$ be the projection map. We write $\mathcal{F} := \pi(A)$ and $\mathcal{G} := \pi(B)$.
Since $\mathcal{F}$ and $\mathcal{G}$ are closed totally geodesic submanifolds of $\Gamma' \backslash (\H^2)^r$ and $\mathcal{F}$ is compact, it follows that $\mathcal{F} \cap \mathcal{G}$ is a compact manifold. In particular, $\mathcal{F} \cap \mathcal{G}$ has only finitely many path-connected components.
Thus, it suffices to show that for all $\gamma_0,\gamma_1 \in \Gamma'$ for which there is a continuous path in $\mathcal{F} \cap \mathcal{G}$ connecting a point in $\pi(\gamma_0 A \cap B)$ to a point in $\pi(\gamma_1 A \cap B)$, we have
\[
\Gamma'_{B}\gamma_0{A} = \Gamma'_{B}\gamma_1{A}.
\]
Let $c \colon [0,1] \to A' \cap B'$ be such a path and choose preimages $x_i \in \gamma_i A \cap B$ with $\pi(x_i) = c(i)$ for $i \in \{0,1\}$.
Since $j_B \colon \Gamma'_{B} \backslash B \to B'$ is a diffeomorphism, $c$ induces a path
\[
c_B := j_B^{-1} \circ c \colon [0,1] \to \Gamma'_{B} \backslash B.
\]
The map $p_B \colon \Gamma' B \to \Gamma'_B \backslash B$, $\gamma \cdot b \mapsto \Gamma_B' b$ is well-defined by \cref{thm:embedding} and is a covering map. By the lifting property of $p_B$ and the fact that $p_B(x_0) = c_B(0)$, there exists a path $\widetilde{c} \colon [0,1] \to \Gamma'B$ with $\widetilde{c}(0) = x_0$ such that the diagram
\[
\begin{tikzcd}[column sep=large]
\Gamma'B \ar{r}{p_B} & \Gamma'_B\backslash B \ar{r}{j_B} & B' \\
{}& {}[0,1]  \ar[dashed]{ul}{\widetilde{c}} \ar{u}{c_B} \ar[swap]{ur}{c}
\end{tikzcd}
\]
commutes. From $c([0,1]) \subset A'$, we deduce that $\widetilde{c}([0,1]) \subset \Gamma' A$. But $\Gamma' A$ is a disjoint union of copies of $A$ by \cref{thm:technical-disjoint-union}, and so $\widetilde{c}(0) = x_0 \in \gamma_0A$ implies that the image of $\widetilde{c}$ must be fully contained in $\gamma_0 A$. In particular, $\widetilde{c}(1) \in \gamma_0 A$.

On the other hand, using $\pi(\widetilde{c}(1)) = c(1) = \pi(x_1)$ and the injectivity of $j_B$, we see that
\[
	\Gamma'_{B} \widetilde{c}(1)
= 	p_B(\widetilde{c}(1))
=	p_B(x_1)
= 	\Gamma'_{B} x_1.
\]
Because of $x_1 \in \gamma_1 A$, this shows that $\widetilde{c}(1) \in \delta \gamma_1 A$ for some $\delta \in \Gamma'_{B}$.
In conclusion, we have $\widetilde{c}(1) \in \gamma_0 A \cap \delta \gamma_1 A$, and so \cref{thm:technical-disjoint-union} implies that $\delta \gamma_1 A = \gamma_0 A$. Hence, by \cref{thm:technical-double-cosets}, we have
\[
\Gamma'_{B}\gamma_0{A} = \Gamma'_{B}\gamma_1{A}.
\tag*{\qedhere}
\]
\end{proof}

\begin{corollary} \label{thm:surjective-map}
Let $\Gamma' \subset \Gamma_{\emb}$ be a subgroup of finite index. Then there exist $\gamma_1,\ldots,\gamma_m \in \Gamma'$ such that the intersection of the images of $A$ and $B$ in $\Gamma' \backslash (\H^2)^r$ is the image of the projection map
\[
\bigcup_{i=1}^m \gamma_i A \cap B \to \Gamma' \backslash (\H^2)^r.
\]
\end{corollary}

\begin{proof}
By \cref{thm:finiteness}, there exist $\gamma_1,\ldots,\gamma_m \in \Gamma'$ with
\begin{equation} \label{eq:projection-surjective-map}
\bigl\{ \Gamma_B'\gamma A : \gamma \in \Gamma' \text{ with } \gamma A \cap B \neq \emptyset \bigr\}
=
\bigl\{
\Gamma'_B\gamma_1 A,
\ldots,
\Gamma'_B\gamma_m A
\bigr\}.
\end{equation}
Let $\pi \colon (\H^2)^r \to \Gamma' \backslash (\H^2)^r$ be the projection map. For each $i \in \{1,\ldots,m\}$, we have $\pi(\gamma_i A \cap B) \subset \pi(A) \cap \pi(B)$. Conversely, let $z \in \pi(A) \cap \pi(B)$. Then there is some $x \in \Gamma' A \cap B$ with $z = \Gamma'x$. Let $\gamma \in \Gamma'$ with $x \in \gamma A$. By \eqref{eq:projection-surjective-map}, we have $\Gamma'_B \gamma A = \Gamma'_B \gamma_i A$ for some $i \in \{1,\ldots,m\}$. So by \cref{thm:technical-double-cosets}, there exists some $\delta \in \Gamma'_B$ with $\gamma A = \delta \gamma_i A$.
Hence, there is some $y \in A$ with $x = \delta \gamma_i y$. From $\gamma_i y = \delta^{-1}x \in B$ and $\Gamma' x = \Gamma' \delta \gamma_i y = \Gamma' \gamma_i y$, we deduce that $z = \Gamma'x \in \pi(\gamma_iA \cap B)$.
\end{proof}

The next two lemmas will be used in the proof of \cref{thm:product-formula-closure}.

\begin{lemma} \label{thm:centralizer-intersection}
For the centralizers of $\alpha$ and $\beta$ in $D \otimes_F \AA{F}$, we have
\[
C_{D \otimes_F \AA{F}}(\beta) \cap C_{D \otimes_F \AA{F}}(\alpha)
= \AA{F}.
\]
\end{lemma}

\begin{proof}
We only show that $C_D(\beta) \cap C_D(\alpha) = F$. This suffices because taking the tensor product with $\AA{F}$ commutes with taking the centralizers.
So suppose to the contrary that there exists $x \in D$ with $x \notin F$, $\alpha x=x \alpha$, and $\beta x=x \beta$.
Note that $\alpha$ and $\beta$ do not commute with each other because otherwise, we would have $A = B$ by \cref{thm:polar-regular-quaternion-algebra}, in contradiction to our assumptions on $A$ and $B$. This implies $\alpha,\beta \notin F$, and so we see that $\{1,x,\alpha,\beta\}$ is a linearly independent subset with four distinct elements of the linear subspace $C_{D}(x) \subset D$. Since $\dim_F(D) = 4$, it follows that $C_D(x) = D$, and so we have $x \in Z(D) = F$. But this contradicts the assumption $x \notin F$.
\end{proof}

To simplify the notation, we will from now on write $C_{G}(g) := \{ h \in G : gh = hg\}$ for a group $G$ whenever multiplication of elements in $G$ with $g$ is defined, even if $g \not\in G$.

\begin{lemma} \label{thm:roots-of-unity-transitively}
The group $\mu_2(\OO{F}) := \{ \omega \in \OO{F} : \omega^2 = 1\}$ of second roots of unity in $\OO{F}$ acts transitively by $\omega \cdot (u,v) := (\omega^{-1} u, \omega v)$ on the fibers of the map
\[
C_{\ag{SL}_\Lambda(\OO{F})}(\beta) \times C_{\ag{SL}_\Lambda(\OO{F})}(\alpha) \to \ag{SL}_\Lambda(\OO{F}),
\quad
(u,v) \mapsto uv.
\]
\end{lemma}

\begin{proof}
Let $(u_1,v_1),(u_2,v_2) \in C_{\ag{SL}_\Lambda(\OO{F})}(\beta) \times C_{\ag{SL}_\Lambda(\OO{F})}(\alpha)$ with $u_1v_1 = u_2v_2$. Then $\omega := u_2^{-1}u_1 = v_2v_1^{-1}$ commutes with both $\alpha$ and $\beta$, and so $\omega$ is a scalar in $\OO{F}$ by \cref{thm:centralizer-intersection}. Since $N(\omega) = 1$ and $N(\omega) = \omega^2$, we have $\omega \in \mu_2(\OO{F})$. It follows that $\omega^{-1} u_1 = u_1 \omega^{-1} = u_2$ and $\omega v_1 = v_2$. Hence we have $\omega \cdot (u_1,v_1) = (u_2,v_2)$.
\end{proof}

Recall from \cref{thm:polar-regular-quaternion-algebra} that $C_{D^\times}(\alpha)$ and $C_{D^\times}(\beta)$ act by orientation-preserving isometries on the flats $A$ and $B$, respectively. The next two propositions will, combined with \cref{thm:surjective-map}, allow us to control the intersection of the images of $A$ and $B$ in a finite covering space of $\Gamma \backslash (\H^2)^r$.

\begin{proposition} \label{thm:product-formula-closure}
For every $\gamma \in \Lambda^1$ that is in the closure of $C_{\Lambda^1}(\beta)C_{\Lambda^1}(\alpha)$ in $\ag{SL}_\Lambda(\OO{F})$, there exist $x \in C_{D^\times}(\beta)$ and $y \in C_{D^\times}(\alpha)$ such that $\gamma = xy$ and such that $x$ acts by orientation-preserving isometries on $(\H^2)^r$.
\end{proposition}

\begin{proof}
\textbf{Step 1:} We first find $x' \in C_{\ag{SL}_\Lambda(\OO{F})}(\beta)$ and $y' \in C_{\ag{SL}_\Lambda(\OO{F})}(\alpha)$ such that $\gamma = x'y'$.
This is possible because $C_{\ag{SL}_\Lambda(\OO{F})}(\beta)$ and $C_{\ag{SL}_\Lambda(\OO{F})}(\alpha)$ are both closed in $\ag{SL}_\Lambda(\OO{F})$, so their product is also closed and contains the set $C_{\Lambda^1}(\beta)C_{\Lambda^1}(\alpha)$, hence also the closure point $\gamma$ of this set.

\textbf{Step 2:}  Next, we find some $c \in \AA{F}$ with $cx' \in D^\times$.
To achieve this, we observe that
$x'$ is a solution in $D \otimes_F \AA{F}$ of the homogeneous system of linear equations
\[
\begin{aligned}
x' \alpha =&\; (\gamma \alpha \gamma^{-1}) x', \\
x' \beta =&\; \beta x'.
\end{aligned}
\]
The coefficients of this system are in $F$. Let $\mathcal{B}$ be an $F$-basis for the space of solutions of this system in $D$. Then the solution space in $D \otimes_F \AA{F}$ is the $\AA{F}$-span of $\mathcal{B}$. In particular, $\mathcal{B} \neq \emptyset$. Moreover, the function $x \mapsto N(x)$ on the solution space in $D$ can be expressed in coordinates with respect to $\mathcal{B}$ by some multivariate polynomial $P \in F[X_1,\ldots,X_m]$. Because of $N(x') \neq 0$, we have $P \neq 0$. So since $F$ is an infinite field, there exists a solution with nonzero reduced norm in $D$, that is, there exists an element $x \in D^\times$ satisfying
\begin{align*}
x \alpha =&\; (\gamma \alpha \gamma^{-1}) x, \\
x \beta =&\; \beta x.
\end{align*}
The element $x^{-1}x' \in D \otimes_F \AA{F}$ commutes with $\beta$. It also commutes with $\alpha$ because from the above two linear systems of equations, we deduce that
\[
x^{-1}x'\alpha = x^{-1} (\gamma \alpha \gamma^{-1}) x' = x^{-1} (x \alpha x^{-1}) x' = \alpha x^{-1} x'.
\]
So by \cref{thm:centralizer-intersection}, we have $x = cx' \in D^\times$ for some $c \in \AA{F}$ as required.

Let $y :=c^{-1}y'$. Then $y = c^{-1}(x')^{-1}\gamma = x^{-1}\gamma \in D^\times$, and so we have $\gamma = xy$ with $x \in C_{D^\times}(\beta)$ and $y \in C_{D^\times}(\alpha)$. It remains to show that $x$ acts by orientation-preserving isometries on $(\H^2)^r$. We do this in the next two steps.

\textbf{Step 3:} We now show that $C_{\Lambda^1}(\beta)x' \cap \mu_2(\OO{F}) U \neq \emptyset$ for every open neighborhood of the identity $U \subset \ag{SL}_\Lambda(\OO{F})$.
Assume to the contrary that $U \subset \ag{SL}_\Lambda(\OO{F})$ is an open neighborhood of the identity with $C_{\Lambda^1}(\beta)x' \cap \mu_2(\OO{F}) U = \emptyset$. Then we have
\begin{equation*}
\bigl(C_{\Lambda^1}(\beta) x' \times y'C_{\Lambda^1}(\alpha)\bigr)
\cap
\bigl(\mu_2(\OO{F}) U \times C_{\ag{SL}_\Lambda(\OO{F})}(\alpha)\bigr) = \emptyset.
\end{equation*}
The set $\mu_2(\OO{F}) U \times C_{\ag{SL}_\Lambda(\OO{F})}(\alpha)$ is invariant under the action of $\mu_2(\OO{F})$ defined in  \cref{thm:roots-of-unity-transitively} and $\mu_2(\OO{F})$ acts transitively on the fibers of the multiplication map by this lemma. Hence, for the images under this map, we obtain
\begin{equation} \label{eq:doublecoset-intersection-pre}
C_{\Lambda^1}(\beta) \gamma C_{\Lambda^1}(\alpha) \cap \mu_2(\OO{F}) U C_{\ag{SL}_\Lambda(\OO{F})}(\alpha) = \emptyset.
\end{equation}
Since $U$ is an open neighborhood of the identity in $\ag{SL}_\Lambda(\OO{F})$, there exists by \cref{thm:adelic-points-basis} a nonzero ideal $\a \subset \O_F$ with $\ag{SL}_\Lambda(\OO{F})(\a) \subset U$. Consequently, we have $\Lambda^1(\a) \subset U$ and so \eqref{eq:doublecoset-intersection-pre} implies that
\begin{equation} \label{eq:doublecoset-intersection}
C_{\Lambda^1}(\beta) \gamma C_{\Lambda^1}(\alpha) \cap \Lambda^1(\a) = \emptyset.
\end{equation}
On the other hand, $\gamma$ is in the closure of $C_{\Lambda^1}(\beta)C_{\Lambda^1}(\alpha)$ and so we have $C_{\Lambda^1}(\beta)C_{\Lambda^1}(\alpha) \cap \Lambda^1(\a)\gamma \neq \emptyset$. Hence, there exist $\gamma_\beta \in C_{\Lambda^1}(\beta)$, $\gamma_\alpha \in C_{\Lambda^1}(\alpha)$ and $\gamma_0 \in \Lambda^1(\a)$ with $\gamma_\beta\gamma_\alpha = \gamma_0 \gamma$.
Since $\Lambda^1(\a)$ is normal in $\Lambda^1$, we obtain
\[
\gamma_\beta^{-1}\gamma_0^{-1}\gamma_\beta = \gamma_\beta^{-1} \gamma \gamma_\alpha^{-1}
\in C_{\Lambda^1}(\beta) \gamma C_{\Lambda^1}(\alpha) \cap \Lambda^1(\a)
\]
in contradiction to \eqref{eq:doublecoset-intersection}. This proves the claim and thus finishes this step.

\textbf{Step 4:}  Finally, we show that $x = cx'$ acts by orientation-preserving isometries on $(\H^2)^r$.
Assume to the contrary that this is not the case. Then there must exist some $i \in \{1,\ldots,r\}$ with $\det(\tau_i(x)) = \sigma_i(N(x)) = \sigma_i(c^2) < 0$. Let $E := F(\sqrt{a})$. We can extend $\sigma_i$ to a real embedding $\widetilde{\sigma_i} \colon E \hookrightarrow \R$ as in the proof of \cref{thm:quaternion-algebra-trick} because by our assumptions we have $\sigma_i(a) > 0$.

Recall that $\tau_1(D^\times) \subset M_2(E)$. So $\tau_1$ induces an $F$-algebra homomorphism $D \to M_2(E)$ and hence also an $F$-homomorphism $\ag{GL}_D \to \Res_{E/F}\ag{GL}_2$. By \cite[p.~15]{Platonov}, we have an isomorphism $(\Res_{E/F}\ag{GL}_2)(\AA{F}) \cong \ag{GL}_2(\AA{E})$, and so we get a continuous group homomorphism
\[
\Phi \colon \ag{GL}_D(\AA{F}) \to \ag{GL}_2(\AA{E})
\]
that extends $\tau_1$ on $D^\times$.
The matrix $\Phi(\beta) = \tau_1(\beta) \in \GL_2(E)$ is by assumption diagonalizable over $E$ with two distinct eigenvalues. So there exists a one-dimensional subspace $L \subset E^2$ which is invariant under $\tau_1(\beta)$. The corresponding eigenspace in $(\AA{E})^2$ of $\tau_1(\beta)$ is $\AA{E}\cdot L$, and so every matrix in $M_2(\AA{E})$ that commutes with $\tau_1(\beta)$ stabilizes $\AA{E}\cdot L$. Let now $v \in L$ be a nonzero vector and let $\ell \in \{1,2\}$ be such that the $\ell$th coordinate of $v$ is $v_\ell \neq 0$. Consider the map
\begin{equation} \label{eq:definition-phi}
s \colon C_{\ag{GL}_D(\AA{F})}(\beta) \to \ag{GL}_1(\AA{E}),
\quad
g \mapsto \biggl(\frac{(\Phi(g)v)_\ell}{v_\ell}\biggr)^2.
\end{equation}
Note that $s$ is multiplicative and so its image is contained in $\ag{GL}_1(\AA{E}) = \AA{E}^\times$. Moreover, $s$ is continuous because $\Phi$ is continuous and $\AA{E}$ is a topological $E$-algebra. We have $s(C_{D^\times}(\beta)) \subset (E^\times)^2$, and so by writing $x' = c^{-1}cx'$, we see that
\[
s(\mu_2(\OO{F}) C_{\Lambda^1}(\beta)x') \subset (E^\times)^2 c^{-2} s(cx')
\]
is contained in $E^\times$ and has only negative images under $\widetilde{\sigma_i}$ because of $\widetilde{\sigma_i}(c^2) < 0$.
Note that $V_+ := \{ v \in \O_E^\times: \widetilde{\sigma_i}(v) > 0 \}$ is a subgroup of finite index in $\O_E^\times$.
So by \cref{thm:chevalley}, there exists a nonzero ideal $\a \subset \O_E$ with $\O_E^\times(\a) \subset V_+$.
Hence $\OO{E}^\times(\a) \cap E^\times \subset \O_E^\times(\a) \subset V_+$, and so we obtain
\[
\mu_2(\OO{F}) C_{\Lambda^1}(\beta)x' \cap s^{-1}(\OO{E}^\times(\a)) = \emptyset.
\]
Since $s$ is continuous, the preimage $s^{-1}(\OO{E}^\times(\a))$ is an open neighborhood of the identity in $C_{\ag{GL}_D(\AA{F})}(\beta)$, and so there exists a nonzero ideal $\b \subset \O_F$ with $C_{\ag{GL}_\Lambda(\OO{F})}(\beta)(\b) \subset s^{-1}(\OO{E}^\times(\a))$.
Let $U := \ag{SL}_\Lambda(\OO{F})(\b)$. Then we have $\mu_2(\OO{F}) C_{\Lambda^1}(\beta)x' \cap U = \emptyset$, or equivalently,
\[
C_{\Lambda^1}(\beta)x' \cap \mu_2(\OO{F}) U = \emptyset.
\]
This contradicts the result from the previous step. So $x$ must act by orientation-preserving isometries on $(\H^2)^r$ and the proof is complete.
\end{proof}

\begin{proposition} \label{thm:product-formula-corollary}
There exists a subgroup of finite index $\Gamma_{\product} \subset \Gamma_{\emb}$ such that every $\gamma \in \Gamma_{\product}$ with $\gamma A \cap B \neq \emptyset$ can be written as $\gamma =xy$ with $x \in C_{D^\times}(\beta)$ and $y \in C_{D^\times}(\alpha)$ so that $x$ acts by orientation-preserving isometries on $(\H^2)^r$.
\end{proposition}

\begin{proof}
By \cref{thm:finiteness}, there exist $\gamma_1,\ldots,\gamma_m \in \Gamma_{\emb}$ with
\[
\bigl\{ (\Gamma_{\emb})_B\gamma A : \gamma \in \Gamma_{\emb},\, \gamma A \cap B \neq \emptyset \bigr\}
=
\bigl\{
(\Gamma_{\emb})_B\gamma_1 A,
\ldots,
(\Gamma_{\emb})_B\gamma_m A
\bigr\}.
\]
We now choose for every $i \in \{1,\ldots,m\}$ a subgroup $\Gamma_{(i)} \subset \Gamma_{\emb}$ as follows:
If $\gamma_i$ is in the closure of $C_{\Lambda^1}(\beta)C_{\Lambda^1}(\alpha)$ in $\ag{SL}_\Lambda(\OO{F})$, then we set $\Gamma_{(i)} := \Gamma_\emb$. Otherwise, there exists by \cref{thm:adelic-points-basis} a nonzero ideal $\a_i \subset \O_F$ with $C_{\Lambda^1}(\beta)C_{\Lambda^1}(\alpha) \cap \ag{SL}_\Lambda(\OO{F})(\a_i)\gamma_i = \emptyset$ and we set $\Gamma_{(i)} := \Gamma_\emb(\a_i)$.
Consider
\[
\Gamma_{\product} := \Gamma_{(1)} \cap \ldots \cap \Gamma_{(m)}.
\]
Let $\gamma \in \Gamma_{\product}$ with $\gamma A \cap B \neq \emptyset$. Then we have $(\Gamma_{\emb})_B \gamma A = (\Gamma_{\emb})_B \gamma_i A$ for some $i \in \{1,\ldots,m\}$. So by \cref{thm:technical-double-cosets}, there exists some $\delta \in (\Gamma_{\emb})_B$ with $\gamma A = \delta \gamma_i A$.
Hence we have $\gamma^{-1} \delta \gamma_i A = A$.
Let $\tau := \gamma^{-1} \delta \gamma_i$. Then $\tau \in (\Gamma_{\emb})_A$ and
\begin{equation} \label{eq:conjugate-equation}
\delta^{-1}\tau = (\delta^{-1}\gamma^{-1}\delta)\gamma_i.
\end{equation}
Since $\Gamma_{\product}$ is normal in $\Gamma_{\emb}$, we have $\delta^{-1}\gamma^{-1}\delta \in \Gamma_{\product}$. Moreover, we have $\delta^{-1} \in (\Gamma_{\emb})_B \subset C_{\Lambda^1}(\beta)$ and $\tau \in (\Gamma_{\emb})_A \subset C_{\Lambda^1}(\alpha)$. Hence, \eqref{eq:conjugate-equation} shows that
\[
C_{\Lambda^1}(\beta)C_{\Lambda^1}(\alpha) \cap \Gamma_{\product}\gamma_i \neq \emptyset.
\]
Now because of $\Gamma_{\product} \subset \Gamma_{(i)}$, we have that $\gamma_i$ must be in the closure of $C_{\Lambda^1}(\beta)C_{\Lambda^1}(\alpha)$ in $\ag{SL}_\Lambda(\OO{F})$. So by \cref{thm:product-formula-closure}, we can write $\gamma_i = x_iy_i$ with $x_i \in C_{D^\times}(\beta)$ and $y_i \in C_{D^\times}(\alpha)$ such that $x_i$ acts by orientation-preserving isometries on $(\H^2)^r$.
Observe that
\[
\gamma = \delta\gamma_i\tau^{-1} = \delta x_i y_i \tau^{-1}.
\]
So we have $\gamma = xy$ for $x := \delta x_i \in C_{D^\times}(\beta)$ and $y := y_i \tau^{-1} \in C_{D^\times}(\alpha)$, and $x$ acts by orientation-preserving isometries on $(\H^2)^r$ because of $N(\delta) = 1$.
\end{proof}

We can now show the following result about the intersection of the images of $A$ and $B$ in a finite covering space of the locally symmetric space $\Gamma \backslash (\H^2)^r$:

\begin{proposition} \label{thm:intersection}
Let $\Gamma' \subset \Gamma_{\product}$ be a subgroup of finite index.
Then the images $\mathcal{F}$ of $A$ and $\mathcal{G}$ of $B$ in $\Gamma' \backslash (\H^2)^r$ intersect if and only if $A$ and $B$ intersect in $(\H^2)^r$. Further, the intersection $\mathcal{F} \cap \mathcal{G}$ is transverse and consists of finitely many points, all of which have the same intersection number.
\end{proposition}

\begin{proof}
We denote as above by $\mathcal{F}$ and $\mathcal{G}$ the images of $A$ and $B$ in $\Gamma' \backslash (\H^2)^r$, respectively.
By \cref{thm:surjective-map}, there exist $\gamma_1,\ldots,\gamma_m \in \Gamma'$ such that the projection map induces a surjection
\[
\bigcup_{i=1}^m \gamma_i A \cap B \twoheadrightarrow \mathcal{F} \cap \mathcal{G}.
\]
Because of $\Gamma' \subset \Gamma_\product$, we can apply \cref{thm:product-formula-corollary} and write each $\gamma_i$ as $\gamma_i = x_iy_i$ for some $x_i \in C_{D^\times}(\beta)$ and $y_i \in C_{D^\times}(\alpha)$ such that $x_i$ acts by orientation-preserving isometries on $(\H^2)^r$.
Now choose orientations $A^+$ on $A$ and $B^+$ on $B$, and let $\mathcal{F}$ and $\mathcal{G}$ carry the induced orientations as described in \cref{rem:induced-orientations}. By \cref{thm:polar-regular-quaternion-algebra}, we have $x_i \cdot B^+ = B^+$ and $y_i \cdot A^+ = A^+$, and so we obtain
\begin{align*}
\gamma_i \cdot A^+ \cap B^+ = x_i \cdot A^+ \cap B^+ = x_i \cdot (A^+ \cap x_i^{-1} \cdot B^+)  = x_i \cdot (A^+ \cap B^+).
\end{align*}
This shows that $\mathcal{F}$ and $\mathcal{G}$ intersect if and only if $A$ and $B$ intersect.
Furthermore, we see that in this case the intersection of $\mathcal{F}$ and $\mathcal{G}$ is transverse and the intersection number is the same in each point of intersection because for each $i \in \{1,\ldots,m\}$, the action of $x_i$ maps the intersection $A^+ \cap B^+$ to the intersection $\gamma_iA^+ \cap B^+$ while also preserving the orientation of the ambient space $(\H^2)^r$.
\end{proof}

\subsection{Finishing the proof of the main theorem} \label{subsec:finishing}

\noindent We now use \cref{thm:intersection} from the previous subsection to show that the flats that we have constructed in \cref{thm:configuration} give us a family of linearly independent real homology classes.
For this, we will use some facts about de Rham cohomology from \cite{Bott-Tu}.

Let $M$ be a smooth oriented $n$-manifold. We write $H_{\dR}^*(M)$ for the de Rham cohomology groups of $M$ and $H_{\dRc}^*(M)$ for the de Rham cohomology groups of $M$ with compact support.

Let $S \subset M$ be a closed oriented $k$-submanifold. We denote by $i \colon S \hookrightarrow M$ the inclusion map. The \term{closed Poincaré dual of $S$} is the unique cohomology class $\eta_S \in H_{\dR}^{n-k}(M)$ so that for all $\omega \in H_{\dRc}^k(M)$, we have
\begin{equation} \label{eq:poincare-dual}
\int_M \omega \wedge \eta_S = \int_S i^*w.
\end{equation}
If additionally $S$ is compact and $M$ is diffeomorphic to the interior of a compact manifold with boundary, then we can also define the \term{compact Poincaré dual of $S$}. This is the unique cohomology class $\eta'_S \in H_{\dRc}^{n-k}(M)$ for which \eqref{eq:poincare-dual} holds with $\eta'_S$ instead of $\eta_S$ and for all $\omega \in H_{\dR}^k(M)$.
Note that the natural map $H_{\dRc}^{n-k}(M) \to H_{\dR}^{n-k}(M)$ sends $\eta'_S$ to $\eta_S$ (see \cite[p.~51]{Bott-Tu}), and that  $\eta'_S$ coincides with the image of the real fundamental class of $S$ under the composition 
\[
H_k(S;\R) \to H_k(M;\R) \congrightarrow H_c^{n-k}(M;\R) \congrightarrow H_{\dRc}^{n-k}(M),
\] 
where the first map is induced by the inclusion $S \hookrightarrow M$, the second map is the Poincaré duality isomorphism and the third map is the de Rham isomorphism for cohomology with compact support.

The Poincaré dual of a transverse intersection of submanifolds is related to the wedge product as follows (see \cite[p.~69]{Bott-Tu}):

\begin{proposition}[Geometric interpretation of the wedge product] \label{thm:wedge-product-intersection}
Let $M$ be a smooth oriented manifold. Then for any two closed oriented submanifolds $S_1$ and $S_2$ of $M$ that intersect transversally, we have
\begin{equation} \label{eq:wedge-product-intersection}
\eta_{S_1} \wedge \eta_{S_1} = \eta_{S_1 \cap S_2}.	
\end{equation}
\end{proposition}

In the above proposition, we equip $S_1 \cap S_2$ with the canonical orientation induced by the orientations on $M$, $S_1$, and $S_2$. If $S_1$ and $S_2$ are of complementary dimensions in $M$, then this orientation is simply given by the intersection numbers of $S_1$ and $S_2$.

We now use this to finish the proof of our main result.
Recall that we call a Riemannian manifold $M$ covered by $(\H^2)^r$ \term{arithmetic} if it is finitely covered by a quotient of $(\H^2)^r$ by an arithmetically defined lattice in $\SL_2(\R)^r$.
The following theorem completes our proof of \cref{thm:main}:

\begin{theorem} \label{thm:main-arithmetic}
Let $M$ be an arithmetic Riemannian manifold covered by $(\H^2)^r$. Then for any $n \in \N$, there exists a connected finite covering $M' \to M$ and compact oriented flat totally geodesic $r$-dimensional submanifolds $\mathcal{F}_1,\ldots,\mathcal{F}_n \subset M'$ such that the images of the fundamental classes $[\mathcal{F}_1],\ldots,[\mathcal{F}_n]$ in $H_r(M';\Z)$ are linearly independent.
\end{theorem}

\begin{proof}
As explained in the beginning of this section, we can and will assume that $M = \Gamma \backslash (\H^2)^r$, where $\Gamma \subset \Lambda^1$ is a torsion-free subgroup of finite index. We assume that all the assumptions introduced there are satisfied.
By \cref{thm:configuration}, there exist maximal flats $A_1,\ldots,A_n$ and $B_1,\ldots,B_n$ in $(\H^2)^r$ such that for all $i,j \in \{1,\ldots,n\}$, \cref{thm:intersection} can be applied to the flats $A = A_i$ and $B = B_j$.
So there exists a subgroup of finite index $\Gamma' \subset \Gamma$ such that the images of $A_1,\ldots,A_n$ and $B_1,\ldots,B_n$ in $M' := \Gamma' \backslash (\H^2)^r$ are closed orientable flat totally geodesic $r$-dimensional submanifolds.
We denote them by $\mathcal{F}_1,\ldots,\mathcal{F}_n$ and $\mathcal{G}_1,\ldots,\mathcal{G}_n$, respectively, and choose orientations on them as in \cref{rem:induced-orientations}.
By \cref{thm:wedge-product-intersection}, we have
\[
\int_{M'} \eta'_{\mathcal{F}_i} \wedge \eta_{\mathcal{G}_j}
= \int_{M'} \eta_{\mathcal{F}_i} \wedge \eta_{\mathcal{G}_j}
= \int_{M'} \eta_{\mathcal{F}_i \cap \mathcal{G}_j}
= \sum_{p \in \mathcal{F}_i \cap \mathcal{G}_j} I_p(\mathcal{F}_i, \mathcal{G}_j).
\]
Furthermore, by what we know about the intersections of $\mathcal{F}_i$ and $\mathcal{G}_j$ from \cref{thm:intersection}, this sum is nonzero if and only if $\mathcal{F}_i \cap \mathcal{G}_j \neq \emptyset$, which is the case if and only if $i \leq j$. It follows that the matrix
\[
\biggl( \int_{M'} \eta'_{\mathcal{F}_i} \wedge \eta_{\mathcal{G}_j} \biggr)_{ij} \!\in M_n(\R)
\]
is upper triangular with nonzero entries on the diagonal, hence is in $\GL_n(\R)$. Since the map $H_{\dRc}^r(M') \times H_{\dR}^{r}(M') \to \R$, $(\omega,\tau) \mapsto \int_{M'} \omega \wedge \tau$ is bilinear, it follows that the cohomology classes $
\eta'_{\mathcal{F}_1}, \ldots, \eta'_{\mathcal{F}_n} \in H_{\dR,c}^{r}(M')$ are linearly independent. They are mapped by the Poincaré duality isomorphism to the images of the real fundamental classes of $\mathcal{F}_1,\ldots,\mathcal{F}_n$ in $H_r(M';\R)$, and so these homology classes are also linearly independent. Thus, the same holds for the images of the integral fundamental classes $[\mathcal{F}_1], \ldots, [\mathcal{F}_n]$ in $H_r(M';\Z)$.
\end{proof}

\section*{Acknowledgments} \label{sec:acknowledgments}

I wish to thank my doctoral advisor Enrico Leuzinger for many helpful comments and discussions. Further, I want to thank Stefan K\"uhnlein for reading an early manuscript of this work and giving important feedback, Roman Sauer, Steffen Kionke, and Benjamin Waßermann for answering questions and helpful discussions, and Aurel Page for bringing my attention to the Grunwald--Wang theorem.

\bibliographystyle{spmpsci}
\bibliography{literature}

\end{document}